\documentclass[12pt]{amsart}

\usepackage{amssymb,amsmath,amsthm,enumerate,graphicx,stmaryrd,tikz}
\usetikzlibrary{arrows}

\newcommand{\ox}{\mathcal{O}_X}

\newcommand{\NR}[1]{N^1(#1)_{\mathbb{R}}}
\newcommand{\NQ}[1]{N^1(#1)_{\mathbb{Q}}}

\def\<{\langle}
\def\>{\rangle}

\def\NN{\mathbb{N}}
\def\OO{\mathcal{O}}
\def\PP{\mathbb{P}}
\def\QQ{\mathbb{Q}}
\def\RR{\mathbb{R}}

\DeclareMathOperator{\codim}{codim}

\DeclareMathOperator{\im}{Im}

\DeclareMathOperator{\ord}{ord}

\DeclareMathOperator{\vol}{vol}

\theoremstyle{plain}
\newtheorem{theorem}{Theorem}[section]
\newtheorem{corollary}[theorem]{Corollary}

\newtheorem{lemma}[theorem]{Lemma}

\newtheorem*{Theorem A}{Theorem A}
\newtheorem*{Theorem B}{Theorem B}
\newtheorem*{Theorem C}{Theorem C}

\theoremstyle{definition}
\newtheorem{notation}[theorem]{Notation}
\newtheorem{definition}[theorem]{Definition}
\newtheorem{remark}[theorem]{Remark}
\newtheorem{example}[theorem]{Example}

\begin{document}

\mbox{}
\vspace{-1.1ex}
\title{Okounkov bodies and restricted volumes along very general curves}
\author{Shin-Yao Jow}
\date{}

\begin{abstract}
 Given a big divisor $D$ on a normal complex projective variety $X$, we show that the restricted
 volume of $D$ along a very general complete-intersection curve $C\subset X$ can be read
 off from the Okounkov body of $D$ with respect to an admissible flag containing $C$.
 From this we deduce that if two big divisors $D_1$ and $D_2$ on $X$ have the same
 Okounkov body with respect to every admissible flag, then $D_1$ and $D_2$ are
 numerically equivalent.
\end{abstract}

\keywords{}

\maketitle

\section*{Introduction}

Motivated by earlier works of Okounkov \cite{Oko96,Oko03},
Lazarsfeld and Musta\c{t}\v{a} gave an interesting construction in a recent paper
\cite{LM}, which associates a convex body $\Delta(D)\subset \RR^d$ to any big divisor $D$
on a projective variety $X$ of dimension $d$. (Independent of \cite{LM}, Kaveh and Khovanskii also came up with a similar construction around the same time: see \cite{KK08,KK09}.) 
This so called ``Okounkov body'' encodes
many asymptotic invariants of the complete linear series $|mD|$ as $m$ goes to infinity.
For example, the \emph{volume} of $D$, which is the limit \[
 \vol_X(D)=\lim_{m\to\infty}\frac{h^0(X,mD)}{m^d/d!}, \]
is equal to $d!$ times the Euclidean volume $\vol_{\RR^d}\bigl(\Delta(D)\bigr)$ of
$\Delta(D)$. This viewpoint renders transparent several basic properties about volumes of
big divisors.

Let us recall now the construction of Okounkov bodies from \cite{LM}. The
construction depends upon the choice of an \emph{admissible flag} on $X$, which is by definition a flag \[
 Y_\bullet \colon X=Y_0\supset Y_1 \supset Y_2 \supset \cdots \supset Y_{d-1} \supset Y_d
 =\{\text{pt}\}  \]
of irreducible subvarieties of $X$, where $\codim_X(Y_i)=i$
and each $Y_i$ is nonsingular at the point $Y_d$. The purpose of this flag is that it
will determine a valuation $\nu_{Y_\bullet}$ which maps any nonzero section $s\in H^0(X,mD)$
to a $d$-tuple of nonnegative integers  \[
     \nu_{Y_\bullet}(s)=(\nu_1(s),\ldots,\nu_d(s))\in \NN^d \]
defined as follows. Assuming that all the $Y_i$'s are smooth after replacing $X$ by an
open subset, we can set to begin with \[
  \nu_1(s)= \ord_{Y_1}(s). \]
After choosing a local equation for $Y_1$ in $X$, $s$ determines a section  \[
  \tilde{s}_1 \in H^0(X,mD-\nu_1(s)Y_1)  \]
that does not vanish identically along $Y_1$, so after restriction we get a nonzero
section \[
  s_1 \in H^0(Y_1,mD-\nu_1(s)Y_1).  \]
Then we set  \[
  \nu_2(s)= \ord_{Y_2}(s_1), \]
and continue in this manner to define the remaining $\nu_i(s)$. Once we have the
valuation $\nu_{Y_\bullet}$, we can define  \[
 \Gamma(D)_m:= \im \bigl( (H^0(X,mD)-\{0\})\xrightarrow{\ \nu_{Y_\bullet}\ }\NN^d\bigr). \]
Then the Okounkov body of $D$ (with respect to the flag $Y_\bullet$) is the compact
convex set  \[
 \Delta(D)=\Delta_{Y_\bullet}(D):= \text{closed convex hull}\ \bigl( \bigcup_{m\ge 1}{\textstyle\frac{1}{m}}\cdot
 \Gamma(D)_m\bigr) \subset \RR^d. \]

Lazarsfeld and Musta\c{t}\v{a} have shown in \cite[Proposition~4.1]{LM} that Okounkov bodies are
numerical in nature, i.e. if $D_1$ and $D_2$ are two numerically equivalent big divisors, then
$\Delta_{Y_\bullet}(D_1)=\Delta_{Y_\bullet}(D_2)$ for every admissible flag $Y_\bullet$.
It is, however, not clear whether one can read off \emph{all} numerical invariants
of a given big divisor from its Okounkov bodies with respect to various flags.
In this paper, we give an affirmative answer to this question when $X$ is normal:

\begin{Theorem A}
 Let $X$ be a normal complex projective variety of dimension $d$. If $D_1$
 and $D_2$ are two big divisors on $X$ such that  \[
   \Delta_{Y_\bullet}(D_1)=\Delta_{Y_\bullet}(D_2) \]
 for every admissible flag $Y_\bullet$ on $X$, then $D_1$ and $D_2$ are numerically
 equivalent.
\end{Theorem A}

We will derive Theorem~A from the following Theorem~B, which says that the restricted
volume of $D$ to a very general complete-intersection curve can be read off from its
Okounkov body:

\begin{Theorem B}
 Let $X$ be a normal complex projective variety of dimension $d$. Let $D$ be a big
 divisor on $X$, and let $A_1,\ldots,A_{d-1}$
 be effective very ample divisors on $X$. If the $A_i$'s are very general, and
 $Y_\bullet$ is an admissible flag such that  \[
       Y_r=A_1\cap\cdots\cap A_r,\quad \forall\, r\in \{1,\ldots,d-1\}, \]
 then the Euclidean volume (length) of \vspace{-10 pt} \[
 \Delta_{Y_\bullet}(D)|_{\mathbf{0}^{d-1}}:=\{x\in \RR \mid
               (\overbrace{0,\ldots,0}^{d-1},x)\in \Delta_{Y_\bullet}(D)\}  \]
 is equal to the restricted volume of $D$ to the curve $Y_{d-1}$, which is the limit
 \[
  \vol_{X|Y_{d-1}}(D)=\lim_{m\to\infty}\frac{\dim(H^0(X,mD)|_{Y_{d-1}})}{m} .\]
\end{Theorem B}

See Theorem~\ref{t:main} for a full statement, including the precise general position condition we need on the very ample divisors $A_i$'s.

Theorem~A will follow from Theorem~B because the difference between $Y_{d-1}\cdot D$ and
$\vol_{X|Y_{d-1}}(D)$ can also be read off from the Okounkov bodies, and that very general
complete-intersection curves are enough to span $N_1(X)_{\RR}$, the dual of the
N\'{e}ron-Severi space $\NR{X}$. Theorem~B in turn is proved by introducing certain
graded linear series $V_\bullet(D;\mathbf{a})$ on $Y_{d-1}$ and studying its asymptotic
behaviors; in particular, we will need a way to compute the volume of
$V_\bullet(D;\mathbf{a})$. This is accomplished by generalizing \cite[Theorem~B]{ELMNP2} which
computes the restricted volume by \emph{asymptotic intersection number}. More precisely,
recall from \cite[Definition~2.5]{LM} that a graded linear series $W_\bullet$ on a variety $X$
is said to satisfy condition~(B) if $W_m\ne 0$ for
all $m\gg 0$, and if for all sufficiently large $m$ the rational map $\phi_m\colon
X\dashrightarrow \PP(W_m)$ defined by $|W_m|$ is birational onto its image. Then we have

\begin{Theorem C}
 Let $X$ be a projective variety of dimension $d$, and let $W_\bullet$ be
 a graded linear series on $X$ satisfying the condition~\textup{(B)} above.
 Fix a positive integer $m>0$ sufficiently large so that the linear series $W_m$ defines a
 birational mapping of $X$, and denote by $B_m=\mathrm{Bs}(W_m)$ the base locus of $W_m$.
 We define the moving self-intersection number $(W_m)^{[d]}$ of $W_m$ by choosing
 $d$ general divisors $D_1,\ldots,D_d\in |W_m|$ and setting \[
   (W_m)^{[d]}:=\#\bigl(D_1\cap\cdots\cap D_d\cap(X-B_m)\bigr). \]
 Then the volume of\/ $W_\bullet$, which is by definition \[
   \vol(W_\bullet):= \lim_{m\to\infty}\frac{\dim(W_m)}{m^d/d!}, \]
 can be computed by the following asymptotic intersection number: \[
  \vol(W_\bullet)=\lim_{m\to\infty}\frac{(W_m)^{[d]}}{m^d}.  \]
\end{Theorem C}

This paper is organized into three sections. In Section~\ref{s:gr lin series} we collect some definitions and notations about graded linear series in general, and define the graded linear series 
$V_\bullet(D;\mathbf{a})$ which will play a key role in the proof of Theorem~B. Then we
study the properties of $V_\bullet(D;\mathbf{a})$ in Section~\ref{s:prop of V}, and use them to
prove the theorems in Section~\ref{s:pf of thm}.

\subsection*{Acknowledgements}
The author would like to thank Robert Lazarsfeld and Mircea Musta\c{t}\v{a} for valuable
discussions and suggestions.

\section{Graded linear series} \label{s:gr lin series}

In this section we first introduce some basic definitions and notations about graded linear series which we will need. Then we define a graded linear series $V_\bullet(D;\mathbf{a})$, which will play a key role in the proof of Theorem~B. We refer the readers to \cite[\S 2.4]{Laz} for more details on graded linear series.

\begin{definition} \label{d:gr lin series}
 Let $X$ be an irreducible variety, and let $L$ be a line bundle on $X$. A
 \emph{graded linear series} on $X$ associated to $L$ consists of a collection \[
    V_\bullet=\{V_m\}_{m\in \NN} \]
 of finite dimensional vector subspaces $V_m\subset H^0(X,L^{\otimes m})$, satisfying \[
    V_k\cdot V_\ell \subset V_{k+\ell} \quad \text{for all }k,\ell\in \NN, \]
 where $V_k\cdot V_\ell$ denotes the image of $V_k\otimes V_\ell$ under the homomorphism
 \[
   H^0(X,L^{\otimes k})\otimes H^0(X,L^{\otimes \ell})\longrightarrow H^0(X,L^{\otimes
   (k+\ell)})  \]
 determined by multiplication. It is also required that $V_0$ contains all constant
 functions.
\end{definition}

\begin{notation}
 Let $X$ be a projective variety and let $L$ be a line bundle on $X$. We write $C_\bullet(X,L)$
 to mean the \emph{complete graded linear series} associated to $L$, namely \[
      C_m(X,L)=H^0(X,L^{\otimes m})\quad \text{for all } m\in \NN. \]
 If $D$ is a Cartier divisor on $X$, we will also write $C_\bullet(X,D)$ for $C_\bullet(X,\ox(D))$.
\end{notation}

\begin{definition}
 Given two graded linear series $V_\bullet$ and $W_\bullet$, we define a
 \emph{morphism} $f_\bullet\colon V_\bullet\to W_\bullet$ of graded linear series to be a
 collection of linear maps $f_m\colon V_m\to W_m$, $m\in \NN$, such that \[
   f_{k+\ell}(s_1\otimes s_2)=f_k(s_1)\otimes f_\ell(s_2) \]
 for all $s_1\in V_k$, $s_2\in V_\ell$ and all $k,\ell\in \NN$. It is also required that
 $f_0$ preserves constant functions.
\end{definition}

\begin{example} \label{e:morphism}
 Let $X$ be a projective variety and let $L$ and $M$ be two line bundles on $X$. If $V_\bullet$ is a
 graded linear series associated to $L$, and $s\in H^0(X,M)$, then the linear maps \[
    V_m \to H^0(X,L^{\otimes m}\otimes M^{\otimes m}), \quad s'\mapsto
   s'\otimes s^{\otimes m} \]
 for all $m\in\NN$ form a morphism from $V_\bullet$ to $C_\bullet(X,L\otimes M)$. We
 will denote this morphism as $\mu(s)_\bullet \colon V_\bullet \to C_\bullet(X,L\otimes M)$.
\end{example}

\begin{definition}
 Let $U_\bullet$, $V_\bullet$ and $W_\bullet$ be graded linear series.
 \begin{enumerate}[(a)]
  \item
   We say that $U_\bullet$ is a \emph{subseries} of $V_\bullet$, denoted by $U_\bullet\subset
   V_\bullet$, if $U_m\subset V_m$ for all $m\in\NN$.
  \item
   If $f_\bullet\colon V_\bullet\to W_\bullet$ is a morphism of graded linear series,
   then the \emph{image} of $f_\bullet$, denoted by $\im(f_\bullet)$, is the subseries of
   $W_\bullet$ consisting of the images of $f_m$ for all $m\in\NN$. If $U_\bullet$ is a
   subseries of $W_\bullet$, then the \emph{preimage}
   of $U_\bullet$ under $f_\bullet$, denoted by $f_\bullet^{-1}(U_\bullet)$, is the
   subseries of $V_\bullet$ consisting of the preimages of $U_m$ under $f_m$ for all
   $m\in\NN$.
  \item
   If $V_\bullet$ is a graded linear series on the variety $X$, and $Y$ is a subvariety of $X$,
   then the \emph{restriction} of $V_\bullet$ to $Y$,
   denoted by $V_\bullet|_Y$, is the graded linear series on $Y$ obtained by restricting
   all of the sections in $V_m$ to $Y$ for all $m\in\NN$.
 \end{enumerate}
\end{definition}

\begin{definition}
 Let $W_\bullet$ be a graded linear series on a variety $X$. The \emph{stable base locus} of $W_\bullet$,  denoted by $\mathbf{B}(W_\bullet)$, is the (set-theoretic) intersection of the base loci $\mathrm{Bs}(W_m)$ for all $m\ge 1$: \[
  \mathbf{B}(W_\bullet):=\bigcap_{m\ge 1} \mathrm{Bs}(W_m). \]
 If $W_\bullet=C_\bullet(X,D)$, the complete graded linear series of a divisor $D$, then we will simply denote its stable base locus by $\mathbf{B}(D)$.
\end{definition}

 It is a simple fact that for any graded linear series $W_\bullet$,
$\mathrm{Bs}(W_m)=\mathbf{B}(W_\bullet)$ for all sufficiently large and divisible $m$ (the proof given in \cite[Proposition~2.1.21]{Laz} for complete graded linear series can be used without change).
Hence $\mathbf{B}(D)$ makes sense even if $D$ is a $\QQ$-divisor.

The stable base locus of a divisor $D$ does not depend only on the numerical equivalence class
of $D$, so it is sometimes preferable to work instead with the \emph{augmented base locus}
$\mathbf{B}_+(D)$, defined as \[
  \mathbf{B}_+(D)=\mathbf{B}(D-A) \]
for any small ample $\QQ$-divisor $A$, this being independent of $A$ provided that it is
sufficiently small.

We will also need the following concept of \emph{asymptotic order of vanishing} which was
studied in \cite{ELMNP1}:

\begin{definition}
 Let $X$ be a projective variety, and let $E$ be a prime Weil divisor not contained in the
 singular locus of $X$. Given a graded linear series $W_\bullet$ on $X$, we define \[
    \ord_{E}(W_m):=\min\{\ord_E(s)\mid s\in W_m \}. \]
 If $W_m\ne 0$ for all $m\gg 0$, then we can define the \emph{asymptotic order of
 vanishing} of $W_\bullet$ along $E$ as \[
    \ord_E(W_\bullet):=\lim_{m\to\infty}\frac{\ord_E(W_m)}{m}. \]
 (cf. \cite[Definition~2.2]{ELMNP1}) When
 $W_\bullet$ is the complete graded linear series associated to a Cartier divisor $D$, we
 will simply write $\ord_E(|mD|)$ for $\ord_{E}(W_m)$, and $\ord_E(\|D\|)$ for
 $\ord_{E}(W_\bullet)$.
\end{definition}

Next we want to recall conditions~(A)--(C) on graded linear series introduced in \cite[\S 2.3]{LM}, which are mild requirements needed for most major statements in that paper.
We have already seen condition~(B) in the Introduction.

\begin{definition} \label{d:conditions}
 Let $W_\bullet$ be a graded linear series on an irreducible variety $X$ of dimension $d$.
 \begin{enumerate}[(a)]
  \item
   We say that $W_\bullet$ satisfies condition~(A) with respect to an admissible flag   $Y_\bullet$ if there is an integer $b>0$ such that for every $0\ne s\in W_m$, \[
    \nu_i(s)\le mb \]
   for all $1\le i \le d$.
  \item
   We say that $W_\bullet$ satisfies condition~(B) if $W_m\ne 0$ for
all $m\gg 0$, and if for all sufficiently large $m$ the rational map \[
 \phi_m\colon X\dashrightarrow \PP(W_m) \]
  defined by $|W_m|$ is birational onto its image.
  \item
   Assume that $X$ is projective, and that $W_\bullet$ is a graded linear series associated to a big divisor $D$. We say that $W_\bullet$ satisfies condition~(C) if:
   \begin{itemize}
    \item
  For every $m\gg 0$ there exists an effective divisor $F_m$ on $X$ such that the divisor
  \[  A_m:=mD-F_m \]
  is ample; and
    \item
  For all sufficiently large $p$, \[
  \quad H^0\bigl(X,\ox(pA_m)\bigr)=H^0\bigl(X,\ox(pmD-pF_m)\bigr)\subset W_{pm}\subset
   H^0\bigl(X,\ox(pmD)\bigr), \]
  where the first inclusion is the natural one determined by $pF_m$.
   \end{itemize}
 \end{enumerate}
\end{definition}

In the proof of \cite[Lemma~1.10]{LM}, it was established that condition~(A) holds automatically if $X$ is projective. Hence condition~(A) is insignificant for us since we will be working with projective varieties. Also note that condition~(C) implies condition~(B). 

We will now define a graded linear series which is particularly relevant to our study of Okounkov
bodies. For the basic setting, let $X$ be an irreducible projective variety of dimension $d$,
and let $A_1,\ldots,A_{d-1}$ be general
effective very ample divisors on $X$. Then by Bertini's theorem, for each $r\in \{1,\ldots,d-1\}$, \[
       Y_r:=A_1\cap\cdots\cap A_r \]
is an irreducible subvariety of $X$ which is of codimension $r$ in $X$ and smooth away
from the singular locus $X_{\text{sing}}$ of $X$. Hence if we let $Y_0:=X$ and let $Y_d$ be a
point in $Y_{d-1}- X_{\text{sing}}$, then \[
  Y_\bullet \colon Y_0\supset Y_1 \supset \cdots \supset Y_{d-1} \supset Y_d \]
is an admissible flag on $X$. We let $\nu_{Y_\bullet}=(\nu_1,\ldots,\nu_d)$ be the
valuation determined by the flag $Y_\bullet$ as described in the Introduction.

\begin{definition} \label{d:V}
 Let $X$ and $Y_\bullet$ be as in the preceding paragraph, and let $D$ be a big Cartier divisor
 on $X$. Given an
 $r$-tuple of nonnegative integers $\mathbf{a}=(a_1,\ldots,a_r)$ where $r\in \{0,\ldots,d-1 \}$,
 the space of sections \[
  \{s\in H^0(X,mD) \mid \nu_i(s)\ge ma_i,\ i=1,\ldots,r \} \subset C_m(X,D), \quad m\in\NN \]
 form a subseries of $C_\bullet(X,D)$, and one can define a morphism from this
 subseries to $C_\bullet(Y_r,D-a_1A_1-\cdots-a_rA_r)$ using an iterated restrictions process similar to the one we saw in the Introduction when we defined the valuation $\nu_{Y_\bullet}$. We then set \[
  V_\bullet(D;\mathbf{a})=V_\bullet(D;a_1,\ldots,a_r) \subset C_\bullet(Y_r,D-a_1A_1-\cdots-a_rA_r) \]
 to be the image of this morphism.

 More pedantically, $V_\bullet(D;\mathbf{a})$ can be defined inductively in the following way.
 If $r=0$, then we simply set \[
     V_\bullet(D;\ ) := C_\bullet(Y_0,D). \]
 Assume that $r>0$ and $V_\bullet(D;a_1,\ldots,a_{r-1})$ has been defined. To define
 $V_\bullet(D;a_1,\ldots,a_{r})$, we first pick a section
 $s_r\in H^0(Y_{r-1},\OO_{Y_{r-1}}(A_r))$ such that $\mathrm{div}(s_r)=Y_r$, and let \[
   \mu(s_r^{\otimes a_r})_\bullet \colon C_\bullet(Y_{r-1},D-a_1A_1-\cdots-a_rA_r)
   \longrightarrow  C_\bullet(Y_{r-1},D-a_1A_1-\cdots-a_{r-1}A_{r-1}) \]
 be the morphism as defined in Example~\ref{e:morphism}. Then we define  \[
  V_\bullet(D;a_1,\ldots,a_{r})
  :=\Bigl(\mu(s_r^{\otimes a_r})_\bullet^{-1}\bigl(V_\bullet(D;a_1,\ldots,a_{r-1})\bigr)\Bigr)\Big|_{Y_r}. \]
 It is obvious that the definition does not depend on the choice of $s_r$.
 It will be convenient to also define $V_\bullet(D;\mathbf{a})$ when
 $a_1,\ldots,a_r$ are nonnegative \emph{rational} numbers in the following way:
 let $\ell$ be the smallest positive integer such that $\ell a_i\in \NN$ for all $i$,
 then set  \[
  V_m(D;a_1,\ldots,a_{r}):=
   \begin{cases}
    V_k(\ell D;\ell a_1,\ldots,\ell a_{r}), &\text{if $m=\ell k$ for some $k\in\NN$;} \\
    0, &\text{otherwise.}
   \end{cases}  \]
 Although in general this is not a graded linear series in the sense of
 Definition~\ref{d:gr lin series}, for all of our purposes we are essentially dealing
 with the graded
 linear series $V_\bullet(\ell D;\ell \mathbf{a})$, so no problem will occur.
\end{definition}

\section{Properties of $V_\bullet(D,\mathbf{a})$} \label{s:prop of V}

In this section we will first explain the motivation behind the construction of
$V_\bullet(D;\mathbf{a})$ in Definition~\ref{d:V}. Then we will show that
$V_\bullet(D;\mathbf{a})$ contains the
restricted complete linear series $C_\bullet(X,D-a_1A_1-\cdots-a_rA_r)|_{Y_r}$. A
consequence of this is that $V_\bullet(D;\mathbf{a})$ satisfies the condition~(C) given in
\cite[Definition~2.9]{LM} as long as $|\mathbf{a}|:=\max\{|a_1|,\ldots,|a_r|\}$ is sufficiently
small, which in turn allows us to compute its volume. We will end
with giving a lower bound to the base locus of $V_m(D;\mathbf{a})$ and its order of vanishing
there. These ingredients will all go into the proof of Theorem~B in
Section~\ref{s:pf of thm}.

As one might already notice, the construction of the graded linear series
$V_\bullet(D;\mathbf{a})$ is closely related to the construction of the Okounkov body of $D$.
Recall from the Introduction that given a $d$-dimensional projective variety $X$
and an admissible flag $Y_\bullet$ on $X$, we get a valuation $\nu_{Y_\bullet}$ which
sends a nonzero global section of any line bundle on $X$ to a $d$-tuple of integers. This
allows us to define the \emph{graded semigroup} of a graded linear series $W_\bullet$ on
$X$ (\cite[Definition~1.15]{LM}): \[
  \Gamma_{Y_\bullet}(W_\bullet):=\{(\nu_{Y_\bullet}(s),m)\mid 0\ne s\in W_m,\, m\ge
  0\}\subset \NN^{d}\times \NN. \]
For any graded semigroup $\Gamma\subset \NN^{d}\times \NN$, a closed convex cone $\Sigma(\Gamma)\subset \RR^d \times \RR$ and a closed convex body
$\Delta(\Gamma)\subset \RR^d$ can be constructed: 
\begin{gather*}
 \Sigma(\Gamma):=\text{the closed convex cone spanned by $\Gamma$}; \\
 \Delta(\Gamma):=\{\mathbf{x}\in\RR^d \mid (\mathbf{x},1)\in \Sigma(\Gamma)\}. 
\end{gather*}
Using these notations, the Okounkov body of a big divisor $D$ on $X$ is \[
  \Delta_{Y_\bullet}(D)=\Delta\Bigl(\Gamma_{Y_\bullet}\bigl(C_\bullet(X,D)\bigr)\Bigr). \]
We will subsequently abbreviate $\Gamma_{Y_\bullet}\bigl(C_\bullet(X,D)\bigr)$ as
$\Gamma_{Y_\bullet}(D)$.

The statement of Theorem~B involves the intersection of $\Delta_{Y_\bullet}(D)$ with the
last coordinate axis, so it is natural to study the following more general intersections:

\begin{notation}
 Given a graded semigroup $\Gamma\subset \NN^d\times \NN$, and an $r$-tuple of nonnegative rational
 numbers $\mathbf{a}=(a_1,\ldots,a_r)$ where $r\le d$, we denote by
 $\Gamma|_{\mathbf{a}}\subset \NN^{d-r}\times \NN$ the graded semigroup \[
  \Gamma|_{\mathbf{a}}:= \{(\nu_{r+1},\ldots,\nu_d,m)\in\NN^{d-r}\times \NN
   \mid (a_1m,\ldots,a_rm,\nu_{r+1},\ldots,\nu_d,m)\in\Gamma\}. \]
 Similarly for a subset $S\subset \RR^d$, we denote by $S|_{\mathbf{a}}\subset\RR^{d-r}$
 the subset \[
  S|_{\mathbf{a}}:=\{(\nu_{r+1},\ldots,\nu_d)\in\RR^{d-r}
           \mid (a_1,\ldots,a_r,\nu_{r+1},\ldots,\nu_d)\in S\}. \]
\end{notation}

\begin{remark} \label{r:Prop A.1}
 Note that in general we have 
 $\Delta(\Gamma|_{\mathbf{a}})\subset \Delta(\Gamma)|_{\mathbf{a}}$.
 If $\Delta(\Gamma)|_{\mathbf{a}}$ meets the interior of $\Delta(\Gamma)$, then
 $\Delta(\Gamma|_{\mathbf{a}})= \Delta(\Gamma)|_{\mathbf{a}}$
 by \cite[Proposition~A.1]{LM}.
\end{remark}

To compute the Euclidean volume of $\Delta_{Y_\bullet}(D)|_{\mathbf{0}^{d-1}}$
in Theorem~B, our plan is to study the Euclidean volume of $\Delta_{Y_\bullet}(D)|_{\mathbf{a}}$
and let $\mathbf{a}$ goes to $\mathbf{0}^{d-1}$. Since \[
 \Delta_{Y_\bullet}(D)|_{\mathbf{a}}=\Delta(\Gamma_{Y_\bullet}(D)|_{\mathbf{a}}) \]
by Remark~\ref{r:Prop A.1}, it is thus desirable to realize the semigroup
$\Gamma_{Y_\bullet}(D)|_{\mathbf{a}}$ as the semigroup of some graded linear series. This
is precisely what motivates the definition of $V_\bullet(D;\mathbf{a})$.

\begin{lemma} \label{l:9}
 Let $X$, $Y_\bullet$, $D$, and $\mathbf{a}$ be as in Definition~\ref{d:V}. Then \[
  \Gamma_{Y_\bullet|_{Y_r}}\bigl(V_\bullet(D;\mathbf{a})\bigr)=\Gamma_{Y_\bullet}(D)|_{\mathbf{a}},\]
 where $Y_\bullet|_{Y_r}$ is the admissible flag $Y_r\supset Y_{r+1}\supset\cdots\supset
 Y_d$ on $Y_r$.
\end{lemma}

\begin{proof}
 $V_\bullet(D;\mathbf{a})$ is defined in the way that makes this true.
\end{proof}

Another important property of $V_\bullet(D;\mathbf{a})$ is that it satisfies the
condition~\textup{(C)} in Definition~\ref{d:conditions} when $|\mathbf{a}|$ is sufficiently small.

\begin{lemma} \label{l:6}
 The graded linear series $V_\bullet(D;a_1,\ldots,a_r)$ satisfies \[
  C_\bullet(X,D-a_1A_1-\cdots-a_rA_r)|_{Y_r}\subset V_\bullet(D;a_1,\ldots,a_r) \subset C_\bullet(Y_r,D-a_1A_1-\cdots-a_rA_r). \]
 (In the case when $a_1,\ldots,a_r$ are rational numbers and $\ell$ is the smallest positive
 integer such that $\ell a_i\in \NN$ for all $i$, we interpret
 $C_\bullet(X,D-a_1A_1-\cdots-a_rA_r)$ in the following manner: \[
  C_m(X,D-a_1A_1-\cdots-a_rA_r):=
   \begin{cases}
    C_k(X,\ell D-\ell a_1A_1-\cdots-\ell a_rA_r), &\text{if $m=\ell k$;} \\
    0, &\text{otherwise.}
   \end{cases}  \]
  And similarly for $C_\bullet(Y_r,D-a_1A_1-\cdots-a_rA_r)$.)
\end{lemma}

\begin{proof}
 It suffices to prove the case when $a_1,\ldots,a_r$ are integers. The containment on the
 right follows from the definition. We show the containment on the left by induction on
 $r$. The case $r=0$ is trivial, so we assume that $r>0$ and \[
  C_\bullet(X,D-a_1A_1-\cdots-a_{r-1}A_{r-1})|_{Y_{r-1}}\subset
  V_\bullet(D;a_1,\ldots,a_{r-1}). \]
 Let $s$ be an arbitrary element in $C_m(X,D-a_1A_1-\cdots-a_{r}A_{r})$, $m\in\NN$. Let
 $s_r\in H^0(X,\ox(A_r))$ be a section whose divisor is $A_r$. Then
 \begin{align*}
  \mu(s_r^{\otimes a_r}|_{Y_{r-1}})_\bullet (s|_{Y_{r-1}})=(s\otimes s_r^{\otimes
  ma_r})|_{Y_{r-1}}\in {}&C_m(X,D-a_1A_1-\cdots-a_{r-1}A_{r-1})|_{Y_{r-1}} \\
    \subset {}&V_m(D;a_1,\ldots,a_{r-1}),
 \end{align*}
 hence $s|_{Y_r}\in V_m(D;a_1,\ldots,a_{r})$ by definition.
\end{proof}

\begin{corollary} \label{c:7}
 The graded linear series $V_\bullet(D;a_1,\ldots,a_r)$ on\/ $Y_r$ satisfies
 condition~\textup{(C)} if\/ $Y_r\nsubseteq \mathbf{B}_+(D)$ and $|\mathbf{a}|$
 is sufficiently small. (Strictly speaking this is an abuse of
 terminology: as remarked toward the end of Definition~\ref{d:V}, what actually satisfies
 condition~\textup{(C)} is $V_\bullet(\ell D;\ell \mathbf{a})$, but this will not cause
 any trouble for us.)

\end{corollary}

\begin{proof}
 This is because under these assumptions the restricted complete linear series contained
 in $V_\bullet(D;\mathbf{a})$ already satisfies condition~(C) by \cite[Lemma~2.16]{LM}.
\end{proof}

Next we want to give a lower bound to the base locus
of $V_m(D;\mathbf{a})$ and its order of
vanishing there.

\begin{lemma} \label{l:base locus V}
 Let $X$, $Y_\bullet$, $D$, and $\mathbf{a}$ be as in Definition~\ref{d:V}. If\/ $Y_{i}$
 intersects $Y_{i-1}\cap \mathbf{B}(D)$ properly in $Y_{i-1}$ for all\/ $i\in\{1,\ldots,r\}$, then \[
    \mathbf{B}\bigl(V_\bullet(D;\mathbf{a})\bigr)\supset \bigl(Y_{r}\cap \mathbf{B}(D)\bigr). \]
 Moreover, if $E$ is a prime Weil divisor of $X$ contained in $\mathbf{B}(D)$, and
 if\/ $F$ is an irreducible component of\/ $Y_{r}\cap E$ such that $F$ is not
 contained in the singular locus of $X$, then \[
    \ord_F\bigl(V_m(D;\mathbf{a})\bigr)\ge \ord_{E}(|mD|),\quad\forall\, m\in\NN. \]
\end{lemma}

\begin{proof}
 We will proceed by induction on $r$. The case $r=0$
 is trivial. Assuming that $r>0$, then by definition any section  \[
      s\in V_m(D;a_1,\ldots,a_r)\subset H^0\bigl(Y_r,m(D-a_1A_1-\cdots-a_rA_r)\bigr) \]
 is the restriction to $Y_r$ of some section \[
      s'\in H^0\bigl(Y_{r-1},m(D-a_1A_1-\cdots-a_rA_r)\bigr) \]
 such that \[
      s'\otimes s_r^{\otimes ma_r}\in V_m(D;a_1,\ldots,a_{r-1})
                 \subset H^0\bigl(Y_{r-1},m(D-a_1A_1-\cdots-a_{r-1}A_{r-1})\bigr), \]
 where $s_r\in H^0(Y_{r-1},A_r)$ is a section whose divisor is $Y_r$. By the induction
 hypothesis, $s'\otimes s_r^{\otimes ma_r}$ vanishes on $Y_{r-1}\cap \mathbf{B}(D)$,
 thus $s'$ vanishes on $Y_{r-1}\cap \mathbf{B}(D)$ since $Y_r$ intersects
 $Y_{r-1}\cap \mathbf{B}(D)$ properly. Hence
 $s=s'|_{Y_r}$ vanishes on $Y_{r}\cap \mathbf{B}(D)$, proving that $Y_{r}\cap
 \mathbf{B}(D)$ is contained in the stable base locus of $V_\bullet(D;a_1,\ldots,a_r)$.
 Let $F'\subset Y_{r-1}$ be an irreducible component of $Y_{r-1}\cap E$ containing $F$.
 Then \[
  \ord_{F'}(s'\otimes s_r^{\otimes ma_r})\ge \ord_{E}(|mD|) \]
 by the induction hypothesis, hence \[
  \ord_{F'}(s')\ge \ord_{E}(|mD|) \]
 since $Y_r$ does not contain $F'$. Therefore \[
  \ord_F(s)=\ord_F(s'|_{Y_r})\ge \ord_{F'}(s')\ge \ord_{E}(|mD|). \]
\end{proof}

\section{Proof of theorems} \label{s:pf of thm}

We start with the proof of Theorem~C:

\begin{proof}[Proof of Theorem~C]
 Since $W_\bullet$ satisfies condition~(B), it must belong to a big divisor $L$.
 Fix a positive integer $m>0$ sufficiently large so that the linear series $W_m$ defines a
 birational mapping of $X$. Let $\pi_m \colon X_m\to X$ be a resolution of the base ideal
 $\mathfrak{b}_m:=\mathfrak{b}(W_m)$. Then we have a decomposition \[
  \pi^*_m |W_m|= |M_m| + F_m, \]
 where $F_m$ is the fixed divisor (i.e. $\mathfrak{b}_m\OO_{X_m}=\OO_{X_m}(-F_m)$),
 and  \[
      M_m\subset H^0(X_m,\pi^*_m L -F_m)  \]
 is a base-point-free linear series. Let $M_{m,\bullet}$ be the graded linear series on
 $X_m$ associated to $\pi^*_m L -F_m$ given by \[
  M_{m,k}:=\im\Bigl(S^k(M_m)\to H^0\bigl(X_m,k(\pi^*_m L -F_m)\bigr) \Bigr). \]
 By Fujita's approximation theorem \cite[Theorem~3.5]{LM}, for any $\epsilon>0$, there
 exists an integer $m_0=m_0(\epsilon)$ such that if $m\ge m_0$, then  \[
   \vol(W_\bullet)-\epsilon\ \le\ \frac{1}{m^d}\cdot\lim_{k\to\infty} \frac{\dim M_{m,k}}{k^d/d!}\ \le
      \ \vol(W_\bullet). \]
 Since we assume that $m$ is sufficiently large so that the morphism defined by the linear series
 $M_m$ maps $X_m$ birationally onto its image $X'_m$ in some $\PP^N$, we have \[
  \lim_{k\to\infty} \frac{\dim M_{m,k}}{k^d/d!}=(\OO_{\PP^N}(1)|_{X'_m})^d=(\pi^*_m L -F_m)^d=(W_m)^{[d]}. \]
 Hence the desired conclusion follows.
\end{proof}

Before we go on to prove Theorem~B, recall that in its statement the very ample divisors
$A_i$'s are required to be \emph{very general}. The precise requirement we will need is
the assumptions on the flag $Y_\bullet$ in Lemma~\ref{l:base locus C} below.

\begin{definition}
 Let $Z$ be a variety, and let $Y\subset Z$ be a prime divisor. Let $C\subset Z$ be a
 closed algebraic subset, and denote by $C_1,\ldots,C_n$ the irreducible components of
 $C$. We say that $Y$ intersects $C$ \emph{very properly} in $Z$ if for every $I\subset \{1,\ldots,n\}$
 such that $\bigcap_{i\in I}C_i \ne \emptyset$, $Y$ does not contain any irreducible
 component of $\bigcap_{i\in I}C_i$.
\end{definition}

\begin{lemma} \label{l:base locus C}
 Let $X$, $Y_\bullet$, and $D$ be as in Definition~\ref{d:V}, and assume that
 $X$ is normal. Let $E_1,\ldots,E_n$ be all of the irreducible $(d-1)$-dimensional
 components of\/ $\mathbf{B}(D)$. Assume that $Y_{i}$ intersects
 $Y_{i-1}\cap \mathbf{B}(D)$ very properly in $Y_{i-1}$ for all $i\in\{1,\ldots,d\}$.
 Then \[
   \mathbf{B}(C_\bullet(X,D)|_{Y_{d-1}})= \bigl(Y_{d-1}\cap \mathbf{B}(D)\bigr)
    =\coprod_{i=1}^n(Y_{d-1}\cap E_i). \]
 Moreover, let $m\in\NN$ be sufficiently large and divisible so that
 $\mathrm{Bs}(|mD|)=\mathbf{B}(D)$. If, in addition to the above assumptions,
 the curve $Y_{d-1}$ intersects each of the $E_i$'s transversally at smooth points
 of\/ $X$, and none of these intersection points lies in an embedded component of the base scheme of
 $|mD|$, then \[
    \ord_p(C_m(X,D)|_{Y_{d-1}})= \ord_{E_i}(|mD|) \]
 for every\/ $p\in Y_{d-1}\cap E_i$, $i\in\{1,\ldots,n\}$.
\end{lemma}

\begin{proof}
 This is obvious.
\end{proof}

\begin{corollary} \label{c:restr vol}
 If the very ample divisors $A_1,\ldots,A_{d-1}$ are very general so that
 Lemma~\ref{l:base locus C} holds for all $m$ such that $\mathrm{Bs}(|mD|)=\mathbf{B}(D)$, then \[
  \vol_{X|Y_{d-1}}(D)=Y_{d-1}\cdot D - \sum_{i=1}^n \sum_{p\in Y_{d-1}\cap E_i}\! \ord_{E_i}\|D\|. \]
\end{corollary}

\begin{proof}
 Let $m$ be sufficiently large and divisible so that $\mathrm{Bs}(|mD|)=\mathbf{B}(D)$.
 By Theorem~C and Lemma~\ref{l:base locus C},
 \begin{align*}
  \vol_{X|Y_{d-1}}(D)&=\lim_{m\to\infty} \frac{Y_{d-1}\cdot(mD)-\sum_{i=1}^n \sum_{p\in Y_{d-1}\cap E_i}
  \ord_{p}(C_m(X,D)|_{Y_{d-1}})}{m} \\
  &= Y_{d-1}\cdot D - \sum_{i=1}^n \sum_{p\in Y_{d-1}\cap E_i} \lim_{m\to\infty}
   \frac{\ord_{E_i}(|mD|)}{m} \\
  &= Y_{d-1}\cdot D - \sum_{i=1}^n \sum_{p\in Y_{d-1}\cap E_i}\! \ord_{E_i}\|D\|.
 \end{align*}
\end{proof}

We will now prove Theorem~B by proving the following more precise statement:

\begin{theorem} \label{t:main}
 Let $X$, $Y_\bullet$, and $D$ be as in Definition~\ref{d:V}. Assume that
 $Y_{r}\nsubseteq \mathbf{B}_+(D)$ for all $r\in\{0,\ldots,d-1\}$.
 \begin{enumerate}[\upshape (a)]
  \item
    If $D$ is ample, then for any\/ $r\in\{0,\ldots,d-1\}$, \[
     \vol_{\RR^{d-r}}\bigl(\Delta_{Y_\bullet}(D)|_{\mathbf{0}^r}\bigr)
      =\frac{\vol_{X|Y_r}(D)}{(d-r)!}=\frac{Y_r\cdot D^{d-r}}{(d-r)!}, \]
    where $\mathbf{0}^r$ denotes $(\overbrace{0,\ldots,0}^{r})$.
  \item
    If $D$ is big but not necessarily ample, and assume that $X$ is normal and that the very ample
    divisors $A_1,\ldots,A_{d-1}$ are very general so that Lemma~\ref{l:base locus C} and
    Corollary~\ref{c:restr vol}
    hold, then the first equality in \textup{(a)} still holds when $r=d-1$, i.e. \[
     \vol_{\RR^{1}}\bigl(\Delta_{Y_\bullet}(D)|_{\mathbf{0}^{d-1}}\bigr)=\vol_{X|Y_{d-1}}(D).\]
 \end{enumerate}
\end{theorem}

\begin{proof}
 Since $Y_{d-1}\nsubseteq \mathbf{B}_+(D)$, we have $\vol_{X|Y_{d-1}}(D)>0$. Hence
 \begin{align*}
  \vol_{\RR^{1}}\bigl(\Delta_{Y_\bullet}(D)|_{\mathbf{0}^{d-1}}\bigr)
   =\vol_{\RR^{1}}\bigl(\Delta\bigl(\Gamma_{Y_\bullet}(D)\bigr)|_{\mathbf{0}^{d-1}}\bigr)
   &\ge\vol_{\RR^{1}}\bigl(\Delta(\Gamma_{Y_\bullet}(D)|_{\mathbf{0}^{d-1}})\bigr)\\
   &=\vol_{X|Y_{d-1}}(D)>0.
 \end{align*}
 This implies we can find $\mathbf{a}\in\QQ_+^r$ with arbitrarily small norm $|\mathbf{a}|$
 such that
 $\Delta_{Y_\bullet}(D)|_{\mathbf{a}}$ meets the interior of $\Delta_{Y_\bullet}(D)$.
 Hence by \cite[Proposition~A.1]{LM}, \[
  \Delta_{Y_\bullet}(D)|_{\mathbf{a}}=\Delta\bigl(\Gamma_{Y_\bullet}(D)\bigr)|_{\mathbf{a}}
   =\Delta\bigl(\Gamma_{Y_\bullet}(D)|_{\mathbf{a}}\bigr)=\Delta\Bigl(\Gamma_{Y_\bullet|_{Y_r}}\bigl(V_\bullet(D;\mathbf{a})\bigr)\Bigr),
   \]
 where the last equality follows from Lemma~\ref{l:9}. By Corollary~\ref{c:7},
 $V_\bullet(D;\mathbf{a})$ satisfies condition~(C) as long as $|\mathbf{a}|$ is sufficiently
 small, hence we can calculate the Euclidean volume of
 $\Delta\Bigl(\Gamma_{Y_\bullet|_{Y_r}}\bigl(V_\bullet(D;\mathbf{a})\bigr)\Bigr)=\Delta_{Y_\bullet|_{Y_r}}\bigl(V_\bullet(D;\mathbf{a})\bigr)$
 by \cite[Theorem~2.13]{LM}:  \[
 \vol_{\RR^{d-r}}\Bigl(\Delta_{Y_\bullet|_{Y_r}}\bigl(V_\bullet(D;\mathbf{a})\bigr)\Bigr)=\frac{\vol\bigl(V_\bullet(D;\mathbf{a})\bigr)}{(d-r)!}.
 \]
 Combining the above equalities gives us
 \begin{equation} \label{e:*}
  \vol_{\RR^{d-r}}\bigl(\Delta_{Y_\bullet}(D)|_{\mathbf{a}}\bigr)=
    \frac{\vol\bigl(V_\bullet(D;\mathbf{a})\bigr)}{(d-r)!}.  \tag{$\ast$}
 \end{equation}

 If $D$ is ample, then $D-\sum_{i=1}^r a_iA_i$ is also ample as long as $|\mathbf{a}|$ is
 sufficiently small, and hence \[
   C_m(X,D-\sum_{i=1}^r a_iA_i)\Bigr|_{Y_r}= C_m(Y_r,D-\sum_{i=1}^r a_iA_i) \]
 for all sufficiently large $m$. So by Lemma~\ref{l:6}, \[
  \vol\bigl(V_\bullet(D;\mathbf{a})\bigr)= \vol\bigl(C_\bullet(Y_r,D-\sum_{i=1}^r a_iA_i)\bigr)
    = Y_r \cdot (D-\sum_{i=1}^r a_iA_i)^{d-r}. \]
 Substituting this back into \eqref{e:*}, we get   \[
  \vol_{\RR^{d-r}}\bigl(\Delta_{Y_\bullet}(D)|_{\mathbf{a}}\bigr)=\frac{Y_r \cdot (D-\sum_{i=1}^r
  a_iA_i)^{d-r}}{(d-r)!}, \]
 and letting $\mathbf{a}$ goes to $\mathbf{0}^r$ proves part~(a).

 Now suppose that $D$ is big but not necessarily ample, and $r=d-1$. Then as long as
 $|\mathbf{a}|$ is sufficiently small, $D-\sum_{i=1}^{d-1} a_iA_i$ is also big, and
 $V_\bullet(D;\mathbf{a})$ satisfies condition~(C) by Corollary~\ref{c:7}.
 Let $E_1,\ldots,E_n$ be all of the irreducible $(d-1)$-dimensional components of
 $\mathbf{B}(D)$. By Theorem~C, Lemma~\ref{l:base locus V}, and Lemma~\ref{l:base locus C},
 \begin{align*}
  \vol\bigl(V_\bullet(D;\mathbf{a})\bigr)
  &=\lim_{m\to\infty}\frac{1}{m}\Bigl(Y_{d-1}\cdot m(D-\sum_{i=1}^{d-1}
    a_iA_i)-\sum_{p\in\mathrm{Bs}(V_m(D;\mathbf{a}))}\ord_p\bigl(V_m(D;\mathbf{a})\bigr)\Bigr) \\
  &\le\lim_{m\to\infty}\frac{1}{m}\Bigl(Y_{d-1}\cdot m(D-\sum_{i=1}^{d-1}
    a_iA_i)-\sum_{p\in Y_{d-1}\cap \mathbf{B}(D)}\ord_p\bigl(V_m(D;\mathbf{a})\bigr)\Bigr) \\
  &=\lim_{m\to\infty}\frac{1}{m}\Bigl(Y_{d-1}\cdot m(D-\sum_{i=1}^{d-1}
    a_iA_i)-\sum_{i=1}^n\sum_{p\in Y_{d-1}\cap E_i}\ord_p\bigl(V_m(D;\mathbf{a})\bigr)\Bigr) \\
  &\le\lim_{m\to\infty}\frac{1}{m}\Bigl(Y_{d-1}\cdot m(D-\sum_{i=1}^{d-1}
    a_iA_i)-\sum_{i=1}^n\sum_{p\in Y_{d-1}\cap E_i}\ord_{E_i}(|mD|)\Bigr) \\
  &=Y_{d-1}\cdot(D-\sum_{i=1}^{d-1} a_iA_i)-\sum_{i=1}^n\sum_{p\in Y_{d-1}\cap
     E_i}\ord_{E_i}(\|D\|).
 \end{align*}
 Substituting this back into \eqref{e:*}, we get   \[
  \vol_{\RR^1}\bigl(\Delta_{Y_\bullet}(D)|_{\mathbf{a}}\bigr)\le
    Y_{d-1}\cdot(D-\sum_{i=1}^{d-1} a_iA_i)-\sum_{i=1}^n\sum_{p\in Y_{d-1}\cap
    E_i}\ord_{E_i}(\|D\|), \]
 and letting $\mathbf{a}$ goes to $\mathbf{0}^{d-1}$ gives \[
  \vol_{\RR^1}\bigl(\Delta_{Y_\bullet}(D)|_{\mathbf{0}^{d-1}}\bigr)\le
    Y_{d-1}\cdot D-\sum_{i=1}^n\sum_{p\in Y_{d-1}\cap E_i}\ord_{E_i}(\|D\|)
     = \vol_{X|Y_{d-1}}(D) \]
 by Corollary~\ref{c:restr vol}. Since we saw that
 $\vol_{\RR^1}\bigl(\Delta_{Y_\bullet}(D)|_{\mathbf{0}^{d-1}}\bigr)\ge \vol_{X|Y_{d-1}}(D)$
 in the beginning of the proof, part~(b) is thus established.
\end{proof}

Finally to prove Theorem~A, we need the following lemma:

\begin{lemma} \label{l:mov curve}
 Let $X$ be a smooth projective variety of dimension $d$, and let $Y\subset X$ be a
 transversal complete intersection of\/ $(d-2)$ very ample divisors. If $D_1,\ldots,D_\rho$
 are ample divisors on $X$ whose numerical classes form a basis of $\NQ{X}$, then the
 curve classes \[
               \{C_i:=Y\cdot D_i \mid i=1,\ldots,\rho \}  \]
 form a basis of $N_1(X)_{\QQ}$.
\end{lemma}

\begin{proof}
 By the Lefschetz hyperplane theorem, the numerical classes of $D_1|_Y,\ldots,D_\rho|_Y$
 are linearly independent in $\NQ{Y}$, hence by the Hodge index theorem for surfaces, the
 intersection matrix \[
      (D_i|_Y\cdot D_j|_Y)=(C_i\cdot D_j) \]
 is nondegenerate.
\end{proof}

\begin{proof}[Proof of Theorem~A]
 We first prove the theorem assuming that $X$ is smooth.
 We start with an observation that for any admissible flag $Y_\bullet$ on $X$ and any big divisor $D$
 on $X$, the asymptotic order of vanishing $\ord_{Y_1}(\|D\|)$ equals the minimum of
 the projection of $\Delta_{Y_\bullet}(D)$ to the first coordinate axis. From this we
 see that for every prime divisor $E$ on $X$, \[
   \ord_E(\|D_1\|)=\ord_E(\|D_2\|),  \]
 since we can always extend $E$ into an admissible flag.
 By Lemma~\ref{l:mov curve}, we can choose $\rho$ admissible flags
 $Y^1_\bullet,\ldots,Y^\rho_\bullet$, such that each one is sufficiently general to make
 Corollary~\ref{c:restr vol} and Theorem~\ref{t:main}~(b) hold, and the numerical classes of
 the curves $Y^1_{d-1},\ldots,Y^\rho_{d-1}$ form a basis of $N_1(X)_{\QQ}$. It then
 follows that $Y^i_{d-1}\cdot D_1=Y^i_{d-1}\cdot D_2$ for all $i\in\{1,\ldots,\rho\}$,
 hence $D_1$ and $D_2$ are numerically equivalent.

 When $X$ is just normal but not smooth, we let $\pi\colon X'\to X$ be a resolution of
 singularities which is an isomorphism over $X-X_{\text{sing}}$. Then \[
   H^0(X,D)=H^0(X',\pi^*D) \]
 for any big divisor $D$ on $X$, and therefore if $Y'_\bullet$ is an admissible flag on $X'$
 such that $\pi(Y'_d)\notin X_{\text{sing}}$, then \[
   \Delta_{Y'_\bullet}(\pi^*D)=\Delta_{\pi(Y'_\bullet)}(D). \]
 It thus follows from the previous paragraph on the smooth case that $\pi^*D_1$ and $\pi^*D_2$ are
 numerically equivalent, and hence so are $D_1$ and $D_2$.
\end{proof}


\end{document}